\documentclass[12pt]{amsart}
\usepackage{amscd,amsmath,amsthm,amssymb}
\usepackage{pstcol,pst-plot,pst-3d}
\usepackage{color}
\usepackage{pstricks}
\usepackage{stmaryrd}
\usepackage{tikz}
\usepackage{url}

\usepackage{latexsym}
\usepackage{amsfonts,amsmath,mathtools}
\usepackage{graphics}
\usepackage{float}
\usepackage{enumitem}

\newpsstyle{fatline}{linewidth=1.5pt}
\newpsstyle{fyp}{fillstyle=solid,fillcolor=verylight}
\definecolor{verylight}{gray}{0.97}
\definecolor{light}{gray}{0.9}
\definecolor{medium}{gray}{0.85}
\definecolor{dark}{gray}{0.6}

\def\NZQ{\mathbb}               

\def\ZZ{{\NZQ Z}}

\def\FF{{\NZQ F}}

\def\G{{\mathcal G}}

\def\HS{\textup{HS}}
\def\pd{\textup{proj}\phantom{.}\!\textup{dim}}

\def\opn#1#2{\def#1{\operatorname{#2}}} 
\opn\chara{char} \opn\length{\ell} \opn\pd{pd} \opn\rk{rk}
\opn\projdim{proj\,dim} \opn\injdim{inj\,dim} \opn\rank{rank}
\opn\depth{depth} \opn\grade{grade} \opn\height{height}
\opn\embdim{emb\,dim} \opn\codim{codim}

\opn\Tr{Tr} \opn\bigrank{big\,rank}
\opn\superheight{superheight}\opn\lcm{lcm}
\opn\trdeg{tr\,deg}
\opn\reg{reg} \opn\lreg{lreg} \opn\ini{in} \opn\lpd{lpd}
\opn\size{size} \opn\sdepth{sdepth}
\opn\link{link}\opn\fdepth{fdepth}\opn\lex{lex}
\opn\tr{tr}
\opn\type{type}
\opn\gap{gap}
\opn\diam{diam}
\opn\Mod{Mod}
\opn\div{div} \opn\Div{Div} \opn\cl{cl} \opn\Cl{Cl}
\opn\Spec{Spec} \opn\Supp{Supp} \opn\supp{supp} \opn\Sing{Sing}
\opn\Ass{Ass} \opn\Min{Min}\opn\Mon{Mon}
\opn\Ann{Ann} \opn\Rad{Rad} \opn\Soc{Soc}
\opn\Im{Im} \opn\Ker{Ker} \opn\Coker{Coker} \opn\Am{Am}
\opn\Hom{Hom} \opn\Tor{Tor} \opn\Ext{Ext} \opn\End{End}
\opn\Aut{Aut} \opn\id{id}

\opn\nat{nat}
\opn\pff{pf}
\opn\Pf{Pf} \opn\GL{GL} \opn\SL{SL} \opn\mod{mod} \opn\ord{ord}
\opn\Gin{Gin} \opn\Hilb{Hilb}\opn\sort{sort}
\opn\PF{PF}\opn\Ap{Ap}
\opn\dist{dist}
\opn\aff{aff}
\opn\relint{relint} \opn\st{st}
\opn\lk{lk} \opn\cn{cn} \opn\core{core} \opn\vol{vol}  \opn\inp{inp} \opn\nilpot{nilpot}
\opn\link{link} \opn\star{star}\opn\lex{lex}\opn\set{set}
\opn\width{wd}
\opn\Fr{F}
\opn\QF{QF}
\opn\G{G}
\opn\type{type}\opn\res{res}
\opn\conv{conv}
\opn\sr{sr}
\opn\gr{gr}

\def\pot#1#2{#1[\kern-0.28ex[#2]\kern-0.28ex]}

\opn\dirlim{\underrightarrow{\lim}}
\opn\inivlim{\underleftarrow{\lim}}
\def\Implies{\ifmmode\Longrightarrow \else
	\unskip${}\Longrightarrow{}$\ignorespaces\fi}
\def\implies{\ifmmode\Rightarrow \else
	\unskip${}\Rightarrow{}$\ignorespaces\fi}
\def\iff{\ifmmode\Longleftrightarrow \else
	\unskip${}\Longleftrightarrow{}$\ignorespaces\fi}

\let\:=\colon
\newtheorem{Theorem}{Theorem}[section]
\newtheorem{Lemma}[Theorem]{Lemma}
\newtheorem{Corollary}[Theorem]{Corollary}
\newtheorem{Proposition}[Theorem]{Proposition}
\newtheorem{Remark}[Theorem]{Remark}

\newtheorem{Example}[Theorem]{Example}
\newtheorem{Examples}[Theorem]{Examples}

\let\epsilon\varepsilon
\let\kappa=\varkappa
\def\qed{\ifhmode\textqed\fi
	\ifmmode\ifinner\hfill\quad\qedsymbol\else\dispqed\fi\fi}
\def\textqed{\unskip\nobreak\penalty50
	\hskip2em\hbox{}\nobreak\hfill\qedsymbol
	\parfillskip=0pt \finalhyphendemerits=0}
\def\dispqed{\rlap{\qquad\qedsymbol}}

\opn\dis{dis}
\def\pnt{{\raise0.5mm\hbox{\large\bf.}}}

\opn\Lex{Lex}

\usepackage{lipsum}

\begin{document}

	\title{Dirac's Theorem and Multigraded syzygies}
	\author{Antonino Ficarra, J\"urgen Herzog}
	
	\address{Antonino Ficarra, Department of mathematics and computer sciences, physics and earth sciences, University of Messina, Viale Ferdinando Stagno d'Alcontres 31, 98166 Messina, Italy}
	\email{antficarra@unime.it}
	
	\address{J\"urgen Herzog, Fakult\"at f\"ur Mathematik, Universit\"at Duisburg-Essen, 45117 Essen, Germany}
	\email{juergen.herzog@uni-essen.de}
	
	\thanks{.
	}
	
	
	
	\maketitle
	
	\begin{abstract}
		Let $G$ be a simple finite graph. A famous theorem of Dirac says that $G$ is chordal if and only if $G$ admits a perfect elimination order. It is known by Fr\"oberg that the edge ideal $I(G)$ of $G$ has a linear resolution if and only if the complementary graph $G^c$ of $G$ is chordal. In this article, we discuss some algebraic consequences of Dirac's theorem in the theory of homological shift ideals of edge ideals. Recall that if $I$ is a monomial ideal, $\HS_k(I)$ is the monomial ideal generated by the $k$th multigraded shifts of $I$. We prove that $\HS_1(I)$ has linear quotients, for any monomial ideal $I$ with linear quotients generated in a single degree. For and edge ideal $I(G)$ with linear quotients, it is not true that $\HS_k(I(G))$ has linear quotients for all $k\ge0$. On the other hand, if $G^c$ is a proper interval graph or a forest, we prove that this is the case. Finally, we discuss a conjecture of Bandari, Bayati and Herzog that predicts that if $I$ is polymatroidal, $\HS_k(I)$ is polymatroidal too, for all $k\ge0$. We are able to prove that this conjecture holds for all polymatroidal ideals generated in degree two. 
	\end{abstract}
	
	\section*{Introduction}
	
	Let $S = K[x_1, \ldots, x_n]$ be the standard graded polynomial ring with coefficients in a field $K$ and $G$ be a simple graph on the vertex set $V(G)=\{1,\dots,n\}$ and with edge set $E(G)$. The \textit{edge ideal} of $G$ is the ideal $I(G)$ in $S$ generated by the monomials $x_ix_j$, such that $\{i,j\}\in E(G)$. The classification of all Cohen--Macaulay edge ideals and the classification of all edge ideals with linear resolution are fundamental problems. While the first problem is widely open and considered to be intractable in general, for the second problem we have a complete answer. The \textit{complementary graph} $G^c$ of $G$ is the graph with vertex set $V(G^c)=V(G)$ and where $\{i,j\}$ is an edge of $G^c$ if and only if $\{i,j\}\notin E(G)$. Ralph Fr\"oberg in \cite{Froberg88} proved that $I(G)$ has a linear resolution if and only $G^c$ is \textit{chordal}, that is, it has no induced cycles of length bigger than three. In turn, the classical and fundamental Dirac's theorem on chordal graphs says that a graph $G$ is chordal if and only if $G$ admits a \textit{perfect elimination order} \cite{Dirac61}.
	
	Recently, a new research trend in the theory of monomial ideals has been initiated by the second author, Moradi, Rahimbeigi and Zhu in \cite{HMRZ021a}, see, also, \cite{Bay019,BJT019,CF,F2,F2Pack,HMRZ021b}. For ${\bf a}=(a_1,\dots,a_n)\in\ZZ_{\ge0}^n$, we denote $x_1^{a_1}\cdots x_n^{a_n}$ by ${\bf x^a}$. Let $I\subset S$ be a monomial ideal and let $\FF$ be its minimal multigraded free $S$-resolution. Then, the $k$th free $S$-module in $\FF$ is $F_k=\bigoplus_{j=1}^{\beta_k(I)}S(-{\bf a}_{kj})$, where ${\bf a}_{kj}\in\ZZ_{\ge0}^n$ are the $k$th \textit{multigraded shifts} of $I$. The \emph{$k$th homological shift ideal} of $I$ is the monomial ideal generated by the monomials ${\bf x}^{{\bf a}_{kj}}$ for $j=1,\ldots,\beta_k(I)$. Note that $\HS_0(I)=I$. It is natural to ask what combinatorial and homological properties are satisfied by all $\HS_k(I)$, $k=0, \ldots, \pd(I)$. Any such property is called an \textit{homological shift property} of $I$. If all $\HS_k(I)$ have linear quotients, or linear resolution, we say that $I$ has \textit{homological linear quotients} or \textit{homological linear resolution}, respectively.
	
	In this article we discuss the algebraic consequences of Dirac's theorem on chordal graphs related to the theory of homological shift ideals of edge ideals.
	
	The article is structured as follows. In Section 1, we investigate arbitrary monomial ideals with linear quotients generated in one degree. Our main theorem states that for such an ideal $I$, $\HS_1(I)$ always has linear quotients. The proof relies upon the fact that certain colon ideals are generated by linear forms (Lemma \ref{Lem:equilinquot1}). In particular, $\HS_1(I)$ has a linear resolution. At present we are not able to generalize this result for all monomial ideals with linear resolution. In this case, one could expect even that $\HS_1(I)$ also has linear quotients, if $I$ has a linear resolution. On the other hand, if $I$ is generated in more than one degree, in Example \ref{Ex:HS1notLinQuot} we show that Theorem \ref{Thm:IHS1(I)linQuot} is no longer valid.
	
	Sections 2 and 3 are devoted to homological shifts of edge ideals with linear resolution. Let $G$ be a graph and $I(G)$ be its edge ideal. For unexplained terminology look at Section 2. Unfortunately, even if $I(G)$ has linear resolution, it may not have homological linear resolution in general, (Example \ref{Ex:I(G)NotHLQ}). At present we do not have a complete classification of all edge ideals with homological linear quotients or homological linear resolution. Thus, we determine many classes of cochordal graphs whose edge ideals have homological linear resolution. In particular, for \textit{proper interval graphs} and \textit{forests}, we prove that the edge ideals of their complementary graphs have homological linear quotients, (Theorems \ref {Thm:PropIntGr} and \ref{Thm:HSForests}). For the proof of the first result we introduce the class of \textit{reversible} chordal graphs, and show that any proper interval graph is a reversible graph, (Lemma \ref{Lem:PI=>Rev}). For the second result, we consider two operations on chordal graphs that preserve the homological linear quotients property. Namely, adding whiskers to a chordal graph and taking unions of disjoint chordal graphs, (Propositions \ref{Prop:G+Whisker} and \ref{Prop:HSDisjoint}). Using these results, it is easy to see that $I(G)$ has homological linear quotients, if $G$ is a forest. Indeed, any forest is the union of pairwise disjoint trees, and any tree can be constructed by iteratively adding whiskers to a previously constructed tree on a smaller vertex set.
	
	In the last section, we consider polymatroidal ideals. An equigenerated monomial ideal $I$ is called \textit{polymatroidal} if its minimal set of monomial generators $G(I)$ corresponds to the set of bases of a \textit{discrete polymatroid}, see \cite[Chapter 12]{JT}. Polymatroidal ideals are characterized by the fact that they have linear quotients with respect to the lexicographic order induced by any ordering of the variables. Such characterization is due to Bandari and Rahmati-Asghar \cite{BJT019}. It was conjectured by Bandari, Bayati and Herzog that all homological shift ideals of a polymatroidal ideal are polymatroidal. At present this conjecture is widely open. On the other hand, Bayati proved that the conjecture holds for any squarefree polymatroidal ideal \cite{Bay019}. The second author of this paper, Moradi, Rahimbeigi and Zhu proved that it holds for polymatroidal ideals that satisfy the strong exchange property \cite[Corollary 3.6]{HMRZ021a}; whereas the first author of this paper proved that $\HS_1(I)$ is again polymatroidal if $I$ is such \cite{F2}, pointing towards the validity of the conjecture in general. 
	
	We prove in Theorem \ref{Thm:HSPolyDegree2} that for any polymatroidal ideal $I$ generated in degree two, all homological shift ideals are polymatroidal. In the squarefree case, $I$ may be seen as the edge ideal of a cochordal graph and we apply our criterion on reversibility of perfect elimination orders. Unfortunately our methods are very special and they can not be applied to prove that homological shifts of polymatroidal ideals, generated in higher degree than two, are polymatroidal.
	\section{The first Homological shift of ideals with Linear quotients}
	
	Let $S=K[x_1,\dots,x_n]$ be the standard graded polynomial ring, with $K$ a field. A monomial ideal $I\subset S$ has \textit{linear quotients} if for some ordering $u_1,\dots,u_m$ of its minimal set of monomial generators $G(I)$, all colon ideals $(u_1,\dots,u_{i-1}):u_i$, $i=1,\dots,m$, are generated by variables. We call $u_1,\dots,u_m$ an \textit{admissible order} of $I$. Such order is called \textit{non-increasing} if $\deg(u_1)\le\deg(u_2)\le\dots\le\deg(u_m)$. By \cite[Lemma 2.1]{JahanZheng2010}, an ideal with linear quotients always has a non-increasing admissible order. So, from now, we consider only non-increasing admissible orders.\smallskip
	
	Let $u_1,\dots,u_m$ be an admissible order of an ideal $I\subset S$ having linear quotients. For $i\in\{1,\dots,m\}$, we let
	$$
	\textup{set}(u_i)=\{j:x_j\in(u_1,\dots,u_{i-1}):u_i\}.
	$$
	
	Given a non-empty subset $A$ of $\{1,\dots,n\}$, we set ${\bf x}_A=\prod_{i\in A}x_i$ and ${\bf x}_\emptyset=1$. The multigraded version of \cite[Lemma 1.5]{ET} implies that
	\begin{equation}\label{eq:HSILinQuot}
	\HS_k(I)\ =\ (u_i{\bf x}_A\ :\ i=1,\dots,m,\ A\subseteq\textup{set}(u_i),\ |A|=k).
	\end{equation}
	
	The ideal $(u_1,\dots,u_{i-1}):u_i$ is generated by the monomials $u_j:u_i=\textup{lcm}(u_j,u_i)/u_i$. Hence, $I$ has linear quotients if and only if for all $i=1,\dots,m$ and all $j<i$ there exists $\ell<i$ such that $u_\ell:u_i=x_p$ for some $p$, and $x_p$ divides $u_j:u_i$.
	
	Hereafter, we denote the set $\{1,\dots,n\}$ by $[n]$. For a monomial $u\in S$ and $i\in[n]$, the \textit{$x_i$-degree} of $u$ is the integer $\deg_{x_i}(u)=\max\{j\ge0:x_i^j\ \textup{divides}\ u\}$.
	
	For the proof of our main result we need Corollary \ref{Cor:Lasagna} of the following lemma.
	\begin{Lemma}\label{Lem:equilinquot1}
		Let $I$ be an equigenerated graded ideal with linear relations. Let $f_1,\dots,f_m$ be a minimal set of generators of $I$. Then, for any $1\le i\le m$,
		$$
		(f_1,\dots,f_{i-1},f_{i+1},\dots,f_m):f_i
		$$
		is generated by linear forms.
	\end{Lemma}
	\begin{proof}
		To simplify the notation, we may assume that $i=m$, and we set $J=(f_1,\dots,f_{m-1}):f_m$. Since the $f_i$ are homogeneous elements, $J$ is a graded ideal. Let $r_m\in J$ be an homogeneous element. Then, there exist $r_1,\ldots,r_{m-1}$ such that $r_mf_m=-\sum_{i=1}^{m-1}r_if_i$ with $\deg(r_i)=\deg(r_m)$ for $i=1,\ldots,m-1$. Therefore, $r=(r_1,\ldots,r_m)$ is a homogeneous relation of $I$. By assumption, the relation module of $I$ is generated by linear relations, say $\ell_i=(\ell_{i1},\ldots,\ell_{im})$ for $i=1,\ldots,t$. Therefore, there exist homogeneous elements $s_i\in S$ such that $r=\sum_{i=1}^{t}s_i\ell_i$. This implies that $r_m=\sum_{i=1}^{t}s_i\ell_{i,m}$. Since $\ell_{i,m}\in J$, the desired conclusion follows.
	\end{proof}
\begin{Corollary}\label{Cor:Lasagna}
	Let $I$ be an equigenerated monomial ideal with linear quotients and let $u_1,\dots,u_m$ be its minimal monomial generators. Then, for any $1\le i\le m$,
	$$
	(u_1,\dots,u_{i-1},u_{i+1},\dots,u_m):u_i
	$$
	is generated by variables.
\end{Corollary}
	
	\begin{Theorem}\label{Thm:IHS1(I)linQuot}
		Let $I\subset S$ be an equigenerated monomial ideal having linear quotients. Then $\HS_1(I)$ has linear quotients.
	\end{Theorem}
	\begin{proof}
		We proceed by induction on $m\ge1$. For $m=1$ or $m=2$ there is nothing to prove.
		
		Let $m>2$ and set $J=(u_1,\dots,u_{m-1})$. Let $L=(x_i:i\in\textup{set}(u_m),x_iu_m\notin\HS_1(J))$. Then, by equation (\ref{eq:HSILinQuot}),
		$$
		\HS_1(I)=\HS_1(J)+u_mL.
		$$
		By inductive hypothesis, $\HS_1(J)$ has linear quotients. Let $v_1,\dots,v_r$ be an admissible order of $\HS_1(J)$. If $L=(x_{j_1},\dots,x_{j_s})$, we claim that $v_1,\dots,v_r,x_{j_1}u_m,\dots,x_{j_s}u_m$ is an admissible order of $\HS_1(I)$. We only need to show that
		\begin{equation}\label{eq:colonIdeal}
		(v_1,\dots,v_r,x_{j_1}u_m,\dots,x_{j_{t-1}}u_m):x_{j_t}u_m
		\end{equation}
		is generated by variables, for all $t=1,\dots,s$.
		
		Note that each generator $x_{j_\ell}u_m:x_{j_t}u_m=x_{j_\ell}$, with $\ell<t$ is already a variable. Consider now a generator $v_\ell:x_{j_t}u_m$ for some $\ell=1,\dots,r$. Then $v_\ell=x_hu_j$ for some $j<m$ and $h\in\text{set}(u_j)$. Moreover, we can write $x_{j_t}u_m=x_pu_k$ for some $k<m$.
		
		If $j=k$, then
		$$
		v_\ell:x_{j_t}u_m=x_hu_k:x_pu_k=x_h
		$$
		is a variable and there is nothing to prove.
		
		Suppose now $j\ne k$. Since $u_1,\dots,u_{m-1}$ is an admissible order, by Corollary \ref{Cor:Lasagna}
		$$
		Q=(u_1,\dots,u_{k-1},u_{k+1},\dots,u_{m-1}):u_k
		$$
		is generated by variables. Since $j\ne k$ and $j<m$, $u_j:u_k$ belongs to $Q$. Hence, we can find $b<m$, $b\ne k$ such that $u_b:u_k=x_q$ and $x_q$ divides $u_j:u_k$. Thus $x_qu_k\in\HS_1(J)$.
		
		Note that $x_q$ divides also $x_hu_j:x_pu_k$. Indeed $x_q$ divides $u_j:u_k$. If $x_q$ does not divide $x_hu_j:x_pu_k$, then necessarily $p=q$. But this would imply that $x_{j_t}u_m=x_qu_k\in\HS_1(J)$, against the fact that $x_{j_t}\in L$. Hence $x_q$ divides $x_hu_j:x_pu_k$. But
		$$
		x_qu_k:x_{j_t}u_m=x_qu_k:x_pu_k=x_q
		$$
		belongs to the ideal (\ref{eq:colonIdeal}). Hence $x_hu_j:x_pu_k$ is divided by a variable belonging to the ideal (\ref{eq:colonIdeal}). This concludes our proof.
	\end{proof}
	
	It is natural to ask the following question. Let $I\subset S$ be a monomial ideal having a linear resolution. Is it true that $\HS_1(I)$ has a linear resolution, too?
	
	Theorem \ref{Thm:IHS1(I)linQuot} is no longer valid for monomial ideals with linear quotients generated in more than one degree, as next example of Bayati et all shows \cite{BJT019}.
	\begin{Example}\label{Ex:HS1notLinQuot}
		\rm (\cite[Example 3.3]{BJT019}). Let $I=\left(x_{1}^{2},\,x_{1}x_{2},\,x_{2}^{4},\,x_{1}x_{3}^{4},\,x_{1}x_{3}^{3}x_{4},\,x_{1}x_{3}^{2}x_{4}^{2}\right)$ be an ideal of $S=K[x_1,x_2,x_3,x_4]$. $I$ is a (strongly) stable ideal whose Borel generators are $x_1x_2,x_2^4,x_1x_3^2x_4^2$. It is well-known that stable ideals have linear quotients. Thus $I$ has linear quotients. Using \textit{Macaulay2} \cite{GDS} the package \cite{F2Pack}, we verified that
		\begin{align*}
		\HS_1(I)\ =\ \big(&x_{1}^{2}x_{2},\,x_{1}x_{2}^{4},\,x_{1}x_{3}^{3}x_{4}^{2},\,x_{1}x_{2}x_{3}^{2}x_{4}^{2},\,x_{1}^{2}x_{3}^{2}x_{4}^{2},\,x_{1}x_{3}^{4}x_{4},\\&x_{1}x_{2}x_{3}^{3}x_{4},\,x_{1}^{2}x_{3}^{3}x_{4},\,x_{1}x_{2}x_{3}^{4},\,x_{1}^{2}x_{3}^{4}\big)
		\end{align*}
		has the following Betti table
		$$\begin{array}{c|cccc}
		& 0 & 1 & 2 & 3\\ \hline
		3 & 1 & . & . & .\\
		4 & . & . & . & .\\
		5 & 1 & 1 & . & .\\
		6 & 8 & 15 & 8 & 1\\
		7 & . & . & . & .\\
		8 & . & 3 & 5 & 2
		\end{array}$$
		We show that $\HS_1(I)$ does not have linear quotients. Suppose by contradiction that $\HS_1(I)$ has linear quotients. Then, since the Betti numbers of an ideal with linear quotients do not depend upon the characteristic of the underlying field $K$, we may assume that $K$ has characteristic zero. Hence $\HS_1(I)$ would be componentwise linear, see \cite[Corollary 8.2.21]{JT}. However, this cannot be the case by virtue of \cite[Theorems 8.2.22. and 8.2.23(a)]{JT}. Indeed $\beta_{1,1+8}(\HS_1(I))\ne0$, while $\beta_{0,8}(\HS_1(I))=0$.
	\end{Example}
	
	\section{Homological shifts of proper interval graphs}
	
	Let $G$ be a finite simple graph with vertex set $V(G)=[n]$ and edge set $E(G)$. Let $K$ be a field. The \textit{edge ideal} of $G$ is the squarefree monomial ideal $I(G)$ of $S=K[x_1,\dots,x_n]$ generated by the monomials $x_ix_j$ such that $\{i,j\}\in E(G)$. A graph $G$ is \textit{complete} if every $\{i,j\}$ with $i,j\in[n]$, $i\ne j$, is an edge of $G$. The \textit{open neighbourhood} of $i\in V(G)$ is the set
	$$
	N_G(i)=\big\{j\in V(G):\{i,j\}\in E(G)\big\}.
	$$
	
	A graph $G$ is called \textit{chordal} if it has no induced cycles of length bigger than three. Recall that a \textit{perfect elimination order} of $G$ is an ordering $v_1,\dots,v_n$ of its vertex set $V(G)$ such that $N_{G_i}(v_i)$ induces a complete subgraph on $G_i$, where $G_i$ is the induced subgraph of $G$ on the vertex set $\{i,i+1,\dots,n\}$. Hereafter, if $1,2,\dots,n$ is a perfect elimination order of $G$, we denote it by $x_1>x_2>\dots>x_n$.
	
	\begin{Theorem}
		\textup{(Dirac).} A simple finite graph $G$ is chordal if and only if $G$ admits a perfect elimination order.
	\end{Theorem}
	
	The \textit{complementary graph} $G^c$ of $G$ is the graph with vertex set $V(G^c)=V(G)$ and where $\{i,j\}$ is an edge of $G^c$ if and only if $\{i,j\}\notin E(G)$. A graph $G$ is called \textit{cochordal} if and only if $G^c$ is chordal.
	
	\begin{Theorem}\label{Thm:Froberg}
		\textup{(Fr\"oberg).} Let $G$ be a simple finite graph. Then, $I(G)$ has a linear resolution if and only if $G$ is cochordal.
	\end{Theorem}
	
	It is known by \cite[Theorem 10.2.6]{JT} that $I(G)$ has linear resolution if and only if it has linear quotients. The theorems of Dirac and Fr\"oberg classify all edge ideals with linear quotients. Furthermore if $x_1>x_2>\dots>x_n$ is a perfect elimination order of $G^c$, then $I(G)$ has linear quotients with respect to the lexicographic order $>_{\textup{lex}}$ induced by $x_1>x_2>\dots>x_n$.
	
	Now we turn to the homological shifts of edge ideals with linear quotients. Unfortunately, in general an edge ideal with linear quotients does not even has homological linear resolution as next example shows.
	\begin{Example}\label{Ex:I(G)NotHLQ}
		\rm Let $G$ be the following cochordal graph on six vertices.\bigskip
		
		
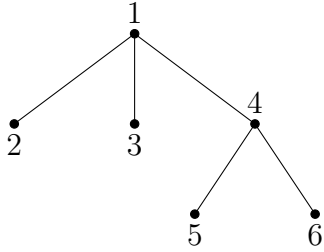
\begin{center}
			\begin{tikzpicture}[scale=0.8]
			\filldraw (4,2) circle (2pt) node[above]{1};
			\filldraw (2,0.5) circle (2pt) node[below]{2};
			\filldraw (4,0.5) circle (2pt) node[below]{3};
			\filldraw (6,0.5) circle (2pt) node[above]{4};
			\filldraw (5,-1) circle (2pt) node[below]{5};
			\filldraw (7,-1) circle (2pt) node[below]{6};
			\draw[-] (4,2) -- (2,0.5);
			\draw[-] (4,2) -- (4,0.5);
			\draw[-] (4,2) -- (6,0.5);
			\draw[-] (6,0.5) -- (5,-1);
			\draw[-] (6,0.5) -- (7,-1);
			\end{tikzpicture}
		\end{center}\vspace{0.2cm}
		
		Let $I=I(G)\subset S=K[x_1,\dots,x_6]$. Using the package \cite{F2Pack} we verified that $\HS_0(I)$ and $\HS_1(I)$ have linear quotients. However the last homological shift ideal $\HS_2(I)=(x_{1}x_{2}x_{3}x_{4},\,x_{1}x_{4}x_{5}x_{6})$ has the following non-linear resolution
		$$
		0\rightarrow S(-6)\rightarrow S(-4)^2\rightarrow (x_1x_2x_3x_4,x_1x_4x_5x_6)\rightarrow0.
		$$
	\end{Example}\medskip
	
	In graph theory, one distinguished class of chordal graphs is the family of \textit{proper interval graphs}. A graph $G$ is called an \textit{interval graph}, if one can label its vertices with some intervals on the real line so that two vertices are adjacent in $G$, when the intersection of their corresponding intervals is non-empty. A \textit{proper interval graph} is an interval graph such that no interval properly contains another.
	
	Now we are ready to state our main result in the section.
	\begin{Theorem}\label{Thm:PropIntGr}
		Let $G$ be a cochordal graph on $[n]$ whose complementary graph $G^c$ is a proper interval graph. Then, $I(G)$ has homological linear quotients.
	\end{Theorem}
	
	In order to prove the theorem we introduce a more general class of graphs.
	
	We call a perfect elimination order $x_1>x_2>\dots>x_n$ of a chordal graph $G$ \textit{reversible} if $x_n>x_{n-1}>\dots>x_1$ is also a perfect elimination order of $G$. We call a chordal graph $G$ \textit{reversible} if $G$ admits a reversible perfect elimination order. Moreover, a cochordal graph $G$ is called reversible if and only if $G^c$ is reversible.
	
	\begin{Lemma}\label{Lem:PI=>Rev}
		Let $G$ be a proper interval graph. Then $G$ is reversible.
	\end{Lemma}
	\begin{proof}
		By \cite[Theorem 1 and Lemma 1]{LO93}, up to a relabeling of the vertex set of $G$, the following property is satisfied:
		\begin{enumerate}
			\item[$(*)$] for all $i<j$, $\{i,j\}\in E(G)$ implies that the induced subgraph of $G$ on $\{i,i+1\dots,j\}$ is a \textit{clique}, \emph{i.e.}, a complete subgraph.
		\end{enumerate}
	With such a labeling, both $x_1>x_2>\dots>x_n$ and $x_n>x_{n-1}>\dots>x_1$ are perfect elimination orders of $G$. By symmetry, it is enough to show that $x_1>x_2>\dots>x_n$ is a perfect elimination order. Let $i\in[n]$, $j,k\in N_G(i)$ with $j,k>i$. We prove that $\{j,k\}\in E(G)$. Suppose $j>k$. By $(*)$, the induced subgraph of $G$ on $\{i,i+1\dots,j\}$ is a clique. Since $j>k>i$, we obtain that $\{j,k\}\in E(G)$, as wanted. 
	\end{proof}
	
	With this lemma at hand, Theorem \ref{Thm:PropIntGr} follows from the following more general result.
	
	\begin{Theorem}\label{Thm:HSI(G)LinQuot}
		Let $G$ be a cochordal graph on $[n]$, and let $x_1>\dots>x_n$ be a reversible perfect elimination order of $G^c$. Then, $\HS_k(I(G))$ has linear quotients with respect to the lexicographic order $>_{\textup{lex}}$ induced by $x_1>\dots>x_n$, for all $k\ge0$.
	\end{Theorem}
	
	For the proof of this theorem, we need a description of the homological shift ideals.
	
	\begin{Lemma}\label{Lemma:setChordal}
		Let $G$ be a cochordal graph on $[n]$, and let $x_1>x_2>\dots>x_n$ be a perfect elimination order of $G^c$. Then, for all $\{i,j\}\in E(G)$, with $i<j$,
		\begin{equation}\label{eq:setI(G)}
		\textup{set}(x_ix_j)=\{1,\dots,i-1\}\cup(\{i+1,\dots,j-1\}\cap N_{G}(i)).
		\end{equation}
		In particular,
		\begin{align*}
		\HS_k(I(G))=\big({\bf x}_A{\bf x}_B\ :&\ A,B\subseteq[n],A,B\ne\emptyset,\max(A)<\min(B),|A\cup B|=k+2,\\
		&\ \{\max(A),b\}\in E(G),\ \text{for all}\ b\in B\big).
		\end{align*}
	\end{Lemma}
	\begin{proof}
		As remarked before $I(G)$ has linear quotients with respect to the lexicographic order $>_{\textup{lex}}$ induced by $x_1>x_2>\dots>x_n$. Let $\{i,j\}\in E(G)$ with $i<j$. Let us determine $\textup{set}(x_ix_j)$. If $k\in\textup{set}(x_ix_j)$, then $x_k(x_ix_j)/x_{\ell}\in I(G)$ and $x_k(x_ix_j)/x_{\ell}>_{\textup{lex}}x_ix_j$ for some $\ell\in\{i,j\}$. Note that $k<j$, indeed for $k>j$, both $x_ix_k,x_jx_k$ are smaller than $x_ix_j$ in the lexicographic order. Thus either $k<i$ or $i<k<j$. We distinguish the two possible cases.
		\smallskip\\
		\textsc{Case 1.} Suppose $k<i$. Assume that none of $x_{k}x_{i},x_{k}x_{j}$ is in $I(G)$. Then $\{k,i\},\{k,j\}\in E(G^c)$. Since $x_1>x_2>\dots>x_n$ is a perfect elimination order, the induced graph of $G_i^c$ on the vertex set $N_{G_k^c}(k)$ is complete. But $i,j>k$ and $i,j\in N_{G_k^c}(k)$. Thus we would have $\{i,j\}\in E(G^c)$, that is, $x_ix_j\notin I(G)$, absurd.
		\smallskip\\
		\textsc{Case 2.} Suppose $i<k<j$. Since $k>i$, $x_kx_j<_{\textup{lex}}x_ix_j$. Thus $k\in\textup{set}(x_ix_j)$ if and only if $x_ix_k\in E(G)$, that is $k\in N_{G}(i)$.
		\medskip\\
		The two cases above show that equation (\ref{eq:setI(G)}) holds. The formula for $\HS_k(I(G))$ follows immediately by applying equations (\ref{eq:HSILinQuot}) and (\ref{eq:setI(G)}).
	\end{proof}
    
    For the proof of the theorem we recall the concept of \textit{Betti splitting} \cite{FHT2009}.
    
    Let $I$, $I_1$, $I_2$ be monomial ideals of $S$ such that $G(I)$ is the disjoint union of $G(I_1)$ and $G(I_2)$. We say that $I=I_1+I_2$ is a \textit{Betti splitting} if
    $$
    \beta_{i,j}(I)=\beta_{i,j}(I_1)+\beta_{i,j}(I_2)+\beta_{i-1,j}(I_1\cap I_2) \ \ \ \textup{for all}\ i,j.
    $$

	\begin{proof}[Proof of Theorem \ref{Thm:HSI(G)LinQuot}.]
		We proceed by induction on $n\ge1$. Let $G'$ be the induced subgraph of $G$ on the vertex set $\{2,3,\dots,n\}$. Then $x_2>x_3>\dots>x_n$ is again a reversible perfect elimination order of $(G')^c$ and $G'$ is a reversible cochordal graph. 
		
		Let $J=(x_i:x_1x_i\in I(G))$. Then, $I(G)=x_1J+I(G')$ is a Betti splitting, because $G(I(G))$ is the disjoint union of $G(x_1J)$ and $G(I(G'))$, and $x_1J$, $I(G')$ have linear resolutions, see \cite[Corollary 2.4]{FHT2009}. Since $I(G')\cap x_1J=x_1I(G')$, \cite[Proposition 1.7]{CF} gives
		$$
		\HS_k(I(G))=x_1\big(\HS_{k-1}(I(G'))+\HS_k(J)\big)+\HS_k(I(G')).
		$$
		
		We claim that $\HS_k(I(G))$ has linear quotients with respect to the lexicographic order $>_{\textup{lex}}$ induced by $x_1>x_2>\dots>x_n$. For $k=0$ this is true. Let $k>0$.\smallskip
		
		Let $u=x_{i_1}x_{j_1}{\bf x}_{F_1},v=x_{i_2}x_{j_2}{\bf x}_{F_2}\in G(\HS_k(I(G)))$, with $u>_{\textup{lex}}v$, $i_1<j_1,i_2<j_2$, $x_{i_1}x_{j_1},\ x_{i_2}x_{j_2}\in I(G)$, $F_1\subseteq\textup{set}(u)$, $F_2\subseteq\textup{set}(v)$. We are going to prove that there exists $w\in G(\HS_k(I(G)))$ such that $w>_{\textup{lex}}v$, $w:v=x_p$ and $x_p$ divides $u:v$.\smallskip
		
		We can write
		$$
		u=x_{p_1}x_{p_2}\cdots x_{p_{k+2}},\ \ \ \ v=x_{q_1}x_{q_2}\cdots x_{q_{k+2}},
		$$
		with $p_1<p_2<\dots<p_{k+2}$, $q_1<q_2<\dots<q_{k+2}$. Since $u>_{\textup{lex}}v$ then $p_1=q_1$, $p_2=q_2$, $\dots$, $p_{s-1}=q_{s-1}$, $p_s<q_s$ for some $s\in\{1,\dots,k+2\}$. If $s=k+2$, then $u:v=x_{p_{k+2}}=x_{j_1}$ and there is nothing to prove. Therefore, we may assume $s<k+2$. Thus $p_s<q_s<q_{k+2}=j_2$. Set $p=p_{s}$ and $q=q_s$, then $x_p$ divides $u:v$.\medskip
		
		Suppose for the moment that $x_1$ divides $v$. Then by definition of $>_{\textup{lex}}$, $p_1=q_1=1$ and $x_1$ divides $u$, too. There are four cases to consider.\medskip\\
		\textsc{Case 1.} Suppose $i_1=i_2=1$. Setting $u'=u/x_1$ and $v'=v/x_1$, we have $u',v'\in G(\HS_k(J))$ and $u'>_{\textup{lex}}v'$. Since $J$ is an ideal generated by variables, it has homological linear quotients with respect to $>_{\textup{lex}}$. Hence, there exists $w'\in G(\HS_k(J))$ with $w'>_{\textup{lex}}v'$ such that $w':v'=x_\ell$ and $x_\ell$ divides $u':v'$. Setting $w=x_1w'$, we have that $w>_{\textup{lex}}v$ and $w\in G(x_1\HS_k(J))\subseteq G(\HS_k(I(G)))$. Hence $w:v=w':v'=x_\ell$ and $x_\ell$ divides $u:v=u':v'$.\medskip\\
		\textsc{Case 2.} Suppose $i_1>1$ and $i_2>1$. Setting $u'=u/x_1$ and $v'=v/x_1$, we have $u',v'\in G(\HS_{k-1}(I(G')))$ and $u'>_{\textup{lex}}v'$. By inductive hypothesis, $I(G')$ has homological linear quotients with respect to $>_{\textup{lex}}'$ induced by $x_2>x_3>\dots>x_n$. Hence, there exists $w'\in G(\HS_{k-1}(I(G')))$ with $w'>_{\textup{lex}}'v'$ such that $w':v'=x_\ell$ and $x_\ell$ divides $u':v'$. Setting $w=x_1w'$, we have that $w>_{\textup{lex}}v$ and $w\in G(x_1\HS_{k-1}(I(G')))\subseteq G(\HS_k(I(G)))$. Hence $w:v=w':v'=x_\ell$ and $x_\ell$ divides $u:v=u':v'$.\medskip\\
		\textsc{Case 3.} Suppose $i_1>1$ and $i_2=1$. Then $1=i_2<p<j_2$.\medskip\\
		\textsc{Subcase 3.1.} Assume $x_1x_p\in I(G)$, then $p\in\textup{set}(x_{i_2}x_{j_2})$. Setting $w=x_{p}(v/x_q)$, by equation (\ref{eq:HSILinQuot}) $w\in G(\HS_k(I(G)))$, and $w>_{\textup{lex}}v$, because $p<q$. Moreover $w:v=x_p$ and $x_p$ divides $u:v$.\medskip\\
		\textsc{Subcase 3.2.} Assume that $x_1x_p\notin I(G)$. By hypothesis, $x_n>x_{n-1}>\dots>x_1$ is also a perfect elimination order of $G^c$. Thus, by Lemma \ref{Lemma:setChordal}, we can write $u={\bf x}_A{\bf x}_B$ with $A=\{p_{k+2},p_{k+1},\dots,p_{r}\}$, $B=\{p_{r-1},\dots,p_{2},p_1\}$ for some $r>1$ and with $\{p_{r},p_\ell\}\in E(G)$ for all $\ell=r-1,\dots,2,1$. Since $\{1,p\}=\{p_1,p_s\}\notin E(G)$, by the above presentation of $u$, $s>r$. Using again Lemma \ref{Lemma:setChordal}, but considering the reversed perfect elimination order $x_n>x_{n-1}>\dots>x_1$, we see that
		\begin{align*}
		w&=x_{q_{s+1}}x_{q_{s+2}}\cdots x_{q_{k+2}}u/(x_{p_{s+1}}x_{p_{s+2}}\cdots x_{p_{k+2}})\\
		&={\bf x}_{(A\setminus\{p_{s+1},p_{s+2},\dots,p_{k+2}\})\cup\{q_{s+1},q_{s+2},\dots,q_{k+2}\}}{\bf x}_B\in G(\HS_k(I(G))).
		\end{align*}
		Moreover, $w>_{\textup{lex}}v$, $w:v=x_p$ and $x_p$ divides $u:v$, as desired.\medskip\\
		\textsc{Case 4.} Suppose $i_1=1$ and $i_2>1$. Recall that $p<j_2$. Moreover $p\ne i_2$, because $x_p$ divides $u:v$ but $x_{i_2}$ divides $v$. Thus there are two cases to consider.\medskip\\
		\textsc{Subcase 4.1.} Assume $p<i_2$. By Lemma \ref{Lemma:setChordal}, $p\in\textup{set}(x_{i_2}x_{j_2})$. If $q\ne i_2$, then $q<j_2$ and by equation (\ref{eq:HSILinQuot}) $w=x_{p}(v/x_q)$ is a minimal generator of $\HS_k(I(G))$. Moreover $w>_{\textup{lex}}v$ and $w:v=x_p$ divides $u:v$, as wanted. Suppose now that $q=i_2$. If there exists $\ell$ such that $x_\ell$ divides $v$ and $i_2<\ell<j_2$, then $\ell>p$ and $w=x_{p}(v/x_\ell)$ is a minimal generator of $\HS_k(I(G))$ such that $w>_{\textup{lex}}v$ and with $w:v=x_p$ dividing $u:v$, as wanted. Otherwise, suppose no such integer $\ell$ exists. Then, $s=k+1$, $q_{k+1}=i_2$ and $q_{k+2}=j_2$. Since $p\in\textup{set}(x_{i_2}x_{j_2})$, then $x_px_{\ell}\in I(G)$, where $\ell\in\{i_2,j_2\}$. Then $p<\ell$ and by Lemma \ref{Lemma:setChordal} we see that $w=x_{p}(v/x_\ell)$ is a minimal generator of $\HS_k(I(G))$ such that $w>_{\textup{lex}}v$ and with $w:v=x_p$ dividing $u:v$.\medskip\\
		\textsc{Subcase 4.2.} Assume now $i_2<p<j_2$. If $x_{i_2}x_p\in I(G)$, by Lemma \ref{Lemma:setChordal}, $p\in\textup{set}(x_{i_2}x_{j_2})$. Setting $w=x_p(v/x_q)$, we have $w\in G(\HS_k(I(G)))$, $w>_{\textup{lex}}v$ and $w:v=x_p$ divides $u:v$. Suppose now that $x_{i_2}x_p\notin I(G)$. By hypothesis, $x_n>x_{n-1}>\dots>x_1$ is also a perfect elimination order of $G^c$. Thus, by Lemma \ref{Lemma:setChordal}, we can write $u={\bf x}_A{\bf x}_B$ with $A=\{p_{k+2},p_{k+1},\dots,p_{r}\}$, $B=\{p_{r-1},\dots,p_{2},p_1\}$ for some $r>1$ and with $\{p_{r},p_\ell\}\in E(G)$ for all $\ell=r-1,\dots,2,1$. Note that $i_2<p$, so $x_{i_2}$ divides $u$. Since $\{i_2,p\}=\{i_2,p_s\}\notin E(G)$, by the above presentation of $u$, $s>r$. Using again Lemma \ref{Lemma:setChordal}, but considering the reversed perfect elimination order $x_n>x_{n-1}>\dots>x_1$, we see that
		\begin{align*}
		w&=x_{q_{s+1}}x_{q_{s+2}}\cdots x_{q_{k+2}}u/(x_{p_{s+1}}x_{p_{s+2}}\cdots x_{p_{k+2}})\\
		&={\bf x}_A{\bf x}_{(B\setminus\{p_{s+1},p_{s+2},\dots,p_{k+2}\})\cup\{q_{s+1},q_{s+2},\dots,q_{k+2}\}}\in G(\HS_k(I(G))).
		\end{align*}
		Moreover, $w>_{\textup{lex}}v$, $w:v=x_p$ and $x_p$ divides $u:v$, as desired.\medskip
		
		Suppose now that $x_1$ does not divide $v$. Then $v\in G(\HS_k(I(G')))$. If $x_1$ does not divide $u$, then $u\in G(\HS_k(I(G')))$, too. Let $>_{\textup{lex}}'$ be the lexicographic order induced by $x_2>x_3>\dots>x_n$. Since by induction $I(G')$ has homological linear quotients with respect to $>_{\textup{lex}}'$ and also $u>_{\textup{lex}}'v$, there exists $w\in G(\HS_k(I(G')))$, with $w>_{\textup{lex}}'v$, $w:v=x_\ell$ and $x_\ell$ divides $u:v$. But also we have $w\in G(\HS_k(I(G)))$ and $w>_{\textup{lex}}v$. Otherwise if $x_1$ divides $u$, then $x_1$ divides $u:v$. Since $\HS_{k}(I(G'))\subseteq\HS_{k-1}(I(G'))$ and $k>0$, we can write $v=x_tw'$ with $w'\in G(\HS_{k-1}(I(G')))$. Let $w=x_1w'$. Then $w>_{\textup{lex}}v$ and $w:v=x_1$ divides $u:v$.\medskip
		
		Hence, the inductive proof is complete and the theorem is proved.
	\end{proof}\medskip

	\begin{Remark}
		\rm Let $x_1>x_2>\dots>x_n$ be a reversible perfect elimination order of $G^c$. By symmetry, Theorem \ref{Thm:HSI(G)LinQuot} shows also that $\HS_k(I(G))$ has linear quotients with respect to the lexicographic order induced by $x_n>x_{n-1}>\dots>x_1$.
	\end{Remark}\medskip
	
	\begin{Examples}\label{Examples:reverseGraphs}
		\rm Let $n,m$ be two positive integers.
		\begin{enumerate}
			\item[(a)] Let $G=K_{n,m}$ be the \textit{complete bipartite graph}. That is, $V(G)=[n+m]$ and $E(G)=\big\{\{i,j\}:i\in[n],j\in\{n+1,\dots,n+m\}\big\}$. For example, for $n=3$ and $m=4$
			
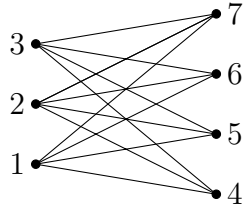
\begin{center}
				\begin{tikzpicture}[scale=0.8]
				\filldraw (0,0) circle (2pt) node[left]{1};
				\filldraw (0,1) circle (2pt) node[left]{2};
				\filldraw (0,2) circle (2pt) node[left]{3};
				\filldraw (3,-0.5) circle (2pt) node[right]{4};
				\filldraw (3,0.5) circle (2pt) node[right]{5};
				\filldraw (3,1.5) circle (2pt) node[right]{6};
				\filldraw (3,2.5) circle (2pt) node[right]{7};
				\draw[-] (0,0) -- (3,-0.5) -- (0,1) -- (3,0.5) -- (0,0) -- (3,1.5) -- (0,2) -- (3,2.5) -- (0,1);
				\draw[-] (0,0) -- (3,2.5) -- (0,1) -- (3,1.5);
				\draw[-] (3,-0.5) -- (0,2) -- (3,0.5);
				\end{tikzpicture}
			\end{center}
			It is easy to see that $G^c$ is the disjoint union of two complete graphs $\Gamma_1$ and $\Gamma_2$ on vertex sets $[n]$ and $\{n+1,\dots,n+m\}$ respectively. Furthermore, any ordering of the vertices is a perfect elimination order of $G^c$. Applying the previous theorem,
			$$
			I(G)=(x_1,\dots,x_n)(x_{n+1},\dots,x_m)
			$$
			has homological linear quotients with respect to the lexicographic order induced by any ordering of the variables.
			\item[(b)] Let $G$ be the graph with vertex set $V(G)=[n+m]$ and edge set
			$$
			E(G)=\big\{\{i,j\}:i\in[n+m],n+1\le j\le n+m,i<j\big\}.
			$$
			We claim that $G$ is a reversible cochordal graph. Indeed $G^c$ is the disjoint union of the complete graph $K_n$ on the vertex set $[n]$ together with the set of isolated vertices $\{n+1,\dots,n+m\}$. It easily seen that any ordering of the vertices is a perfect elimination order of $G^c$. Applying Theorem \ref{Thm:HSI(G)LinQuot}
			$$
			I(G)=(x_1,\dots,x_n)(x_{n+1},\dots,x_m)+(x_ix_j:n+1\le i<j\le n+m)
			$$
			has homological linear quotients with respect to the lexicographic order induced by any ordering of the variables.
		\end{enumerate}
	\end{Examples}

	\section{Homological shifts of Trees}
	
	In this section we construct several classes of edge ideals with homological linear quotients, by considering various operations on cochordal graphs that preserve the homological linear quotients property. As a main application of all these results we will prove the following theorem.
	\begin{Theorem}\label{Thm:HSForests}
		Let $G$ be a graph such that $G^c$ is a forest. Then $I(G)$ has homological linear quotients.
	\end{Theorem}
	
	The \textit{squarefree Veronese ideal $I_{n,d}$ of degree} $d$ in $S=K[x_1,\dots,x_n]$ is the ideal of $S$ generated by all squarefree monomials of degree $d$ in $S$. It is well-known that $I_{n,d}$ has homological linear quotients, (see for instance \cite[Corollary 3.2]{HMRZ021a}).
	
	The first operation we consider consists in adding \textit{whiskers}. Let $\Gamma'$ be a graph on vertex set $[n-1]$. Let $i\in[n-1]$ and let $\Gamma$ be the graph with vertex set $[n]$ and edge set $V(\Gamma)=V(\Gamma')\cup\{\{i,n\}\}$. $\Gamma$ is called the \textit{whisker graph} of $\Gamma'$ obtained by adding the whisker $\{i,n\}$ to $\Gamma'$.
	
	\begin{Proposition}\label{Prop:G+Whisker}
		Let $\Gamma'$ be a graph on vertex set $[n-1]$ and $\Gamma$ be the graph on vertex set $[n]$ and edge set $V(\Gamma)=V(\Gamma')\cup\{\{i,n\}\}$ for some $i\in[n-1]$. Set $G=\Gamma^c$. Suppose $I((\Gamma')^c)$ has homological linear quotients. Then $I(G)$ has homological linear quotients, too.
	\end{Proposition}
	\begin{proof}
		Since $\Gamma'$ is chordal, obviously $\Gamma$ is chordal, too. Set $J=I((\Gamma')^c)$, $I=I(G)$ and $L=(x_j:j\in[n-1]\setminus\{i\})$. Since $N_{G^c}(n)=\{i\}$, we have the Betti splitting:
		\begin{equation}\label{eq:x_n-splitWhisker}
		I=x_nL+J.
		\end{equation}
		Since $G$ is cochordal, $\HS_0(I)$ and $\HS_1(I)$ have linear quotients. So we only have to show that $\HS_k(I)$ has linear quotients for $k\ge2$. By equation (\ref{eq:x_n-splitWhisker}), for all $k\ge2$,
		$$
		\HS_k(I)=x_n\HS_k(L)+x_n\HS_{k-1}(J)+\HS_k(J).
		$$
		Note that $\HS_k(L)$ is the squarefree Veronese ideal of degree $k+1$ in the polynomial ring $K[x_j:j\in[n-1]\setminus\{i\}]$. Thus $\HS_k(L)$ has linear quotients with admissible order, say, $u_1,\dots,u_m$. Let $v_1,\dots,v_r$ and $w_1,\dots,w_s$ be admissible orders of $\HS_{k-1}(J)$ and $\HS_k(J)$, respectively. Let $v_{j_1},\dots,v_{j_p}$, with $j_1<j_2<\dots<j_p$, the monomials in $G(\HS_{k-1}(J))\setminus G(\HS_{k}(L))$. We claim that
		\begin{equation}\label{eq:admOrderWhisker}
		x_{n}u_1,\dots,x_{n}u_m,\ x_{n}v_{j_1},\dots,x_{n}v_{j_p},\ w_1,\dots,w_s
		\end{equation}
		is an admissible order of $\HS_k(J)$.
		
		Let $\ell\in\{1,\dots,m\}$. Then $(x_{n}u_1,\dots,x_{n}u_{\ell-1}):x_{n}u_\ell=(u_1,\dots,u_{\ell-1}):u_{\ell}$ is generated by variables.
		
		Let $\ell\in\{1,\dots,p\}$. We show that
		\begin{align*}
		Q&=(x_{n}u_1,\dots,x_{n}u_m,x_{n}v_{j_1},\dots,x_{n}v_{j_{\ell-1}}):x_{n}v_{j_\ell}\\
		&=(u_1,\dots,u_m,v_{j_1},\dots,v_{j_{\ell-1}}):v_{j_\ell}
		\end{align*}
		is generated by variables. Consider $v_{j_q}:v_{j_\ell}$, then we can find $d<j_\ell$ such that $v_d:v_{j_\ell}$ is a variable that divides $v_{j_q}:v_{j_\ell}$. Either $d=j_b$, for some $b<\ell$, or $v_d\in\HS_{k}(L)$. In any case, $x_nv_d\in(u_1,\dots,u_m,v_{j_1},\dots,v_{j_{\ell-1}})$ and $v_{d}:v_{j_\ell}\in Q$ divides $v_{j_q}:v_{j_\ell}$.\\
		Consider now $u_q:v_{j_\ell}$, $1\le q\le m$. Hence $x_i$ divides $v_{j_\ell}$, lest $v_{j_\ell}\in G(\HS_{k}(L))$. But then $v_{j_\ell}/x_i\in\HS_{k-1}(L)$. Let $x_t$ dividing $u_q:v_{j_\ell}$. Then $u=x_tv_{j_\ell}/x_i\in\HS_{k}(L)$ and $u:v_{j_\ell}=x_t\in Q$ divides $u_q:v_{j_\ell}$.
		
		Finally, let $\ell\in\{1,\dots,s\}$. We show that
		\begin{align*}
		Q&=(x_{n}u_1,\dots,x_{n}u_m,x_{n}v_{j_1},\dots,x_{n}v_{j_p},w_1,\dots,w_{\ell-1}):w_\ell\\
		&=(x_{n}\HS_k(L)+x_{n}\HS_{k-1}(J)):w_\ell+(w_1,\dots,w_{\ell-1}):w_{\ell}
		\end{align*}
		is generated by variables. Since $w_1,\dots,w_{s}$ is an admissible order, $(w_1,\dots,w_{\ell-1}):w_{\ell}$ is generated by variables. Consider now a generator $x_{n}z:w_\ell$ with $z\in\HS_{k}(L)$ or $z\in\HS_{k-1}(J)$. Then $x_n$ divides $x_{n}z:w_\ell$. On the other hand $w_\ell/x_t\in\HS_{k-1}(J)$ for some $t$. But then $x_{n}w_\ell/x_t:w_\ell=x_n\in Q$ divides our generator.
		
		The three cases above show that (\ref{eq:admOrderWhisker}) is an admissible order, as desired. 
	\end{proof}
	
	Since any tree can be constructed iteratively by adding a whisker to a tree on a smaller vertex set at each step, the previous proposition implies immediately
	\begin{Corollary}\label{Cor:HSTree}
		Let $G$ be a graph such that $G^c$ is a tree. Then $I(G)$ has homological linear quotients.
	\end{Corollary}

	The second operation we consider consists in joining disjoint graphs. Two graphs $\Gamma_1$ and $\Gamma_2$ are called \textit{disjoint} if $V(\Gamma_1)\cap V(\Gamma_2)=\emptyset$. The \textit{join} of $\Gamma_1$ and $\Gamma_2$ is the graph $\Gamma$ with vertex set $V(\Gamma)=V(\Gamma_1)\cup V(\Gamma_2)$ and edge set $E(\Gamma)=E(\Gamma_1)\cup E(\Gamma_2)$.
	
	\begin{Proposition}\label{Prop:HSDisjoint}
		Let $\Gamma_1$ and $\Gamma_2$ be disjoint chordal graphs such that $I(\Gamma_1^c),I(\Gamma_2^c)$ have homological linear quotients. Let $\Gamma$ be the join of $\Gamma_1$ and $\Gamma_2$ and set $G=\Gamma^c$. Then $I(G)$ has homological linear quotients, too.
	\end{Proposition}
	\begin{proof}
		Obviously $\Gamma$ is chordal, too. Let $G_1=\Gamma_1^c$, $G_2=\Gamma_2^c$, $V(G_1)=[n]$ and $V(G_2)=\{n+1,\dots,n+m\}$. Set $L=(x_1,\dots,x_n)(x_{n+1},\dots,x_{m})$. Then,
		$$
		I(G)=I(G_1)+I(G_2)+L
		$$
		Suppose $x_1>\dots>x_n$ and $x_{n+1}>\dots>x_{n+m}$ are perfect elimination orders of $\Gamma_1$ and $\Gamma_2$. Then $G=\Gamma^c$ is cochordal. Indeed, $x_1>\dots>x_n>x_{n+1}>\dots>x_{n+m}$ is a perfect elimination order of $\Gamma$. Let $>_\textup{lex}$ be the lexicographic order induced by such an ordering of the variables. Set, $I=I(G)$, $I_1=I(G_1)$ and $I_2=I(G_2)$. Then, $I,I_1,I_2$ and $J$ have linear quotients with respect to $>_{\textup{lex}}$.
		
		Let $k\ge0$ and $u\in G(\HS_k(I))$ such that $x_ix_j$ divides $u$ for some integers $i\in[n]$,\ \ $n+1\!\le\! j\le\! n+m$. We claim that $u\in G(\HS_k(L))$. Let $i_0=\max\{i\in[n]:x_i\ \textup{divides}\ u\}$ and $j_0=\max\{j\in\{n+1,\dots,n+m\}:x_j\ \textup{divides}\ u\}$. Let $u/(x_{i_0}x_{j_0})={\bf x}_F$. Then $F\subseteq\{1,\dots,i_0-1\}\cup\{n+1,\dots,j_0-1\}=\textup{set}_I(x_{i_0}x_{j_0})$ and $x_{i_0}x_{j_0}\in L$. Thus, by equation (\ref{eq:HSILinQuot}), $u=x_{i_0}x_{j_0}{\bf x}_F\in\HS_{k}(L)$, as desired. This argument shows that any squarefree monomial $w\in K[x_1,\dots,x_{n+m}]$ of degree $k+2$, containing as a factor any monomial $x_ix_j$ with $i\in[n]$ and $n+1\le j\le n+m$, is a generator of $\HS_{k}(L)$.
		
		From this remark, for all $k\ge0$, it follows that
		$$
		\HS_{k}(I)=\HS_{k}(L)+\HS_{k}(I_1)+\HS_{k}(I_2).
		$$
		Note that $L$ is the edge ideal of a complete bipartite graph. By Examples \ref{Examples:reverseGraphs}(a), $L$ has homological linear quotients. Let $u_1,\dots,u_m$ be an admissible order of $\HS_k(L)$. Moreover, let $v_1,\dots,v_r$ and $w_1,\dots,w_s$ be admissible orders of $\HS_k(I_1)$ and $\HS_{k}(I_2)$, respectively. Note that the monomials $u_i,v_j,w_t$ are all different, because all monomials $u_i$ contain a factor $x_{i_0}x_{j_0}$ with $i_0\in[n]$ and $j_0\in\{n+1,\dots,n+m\}$. Whereas, the $v_j$ are monomials in $K[x_1,\dots,x_n]$ and the $w_t$ are monomials in $K[x_{n+1},\dots,x_{n+m}]$.\smallskip
		
		We claim that
		\begin{equation}\label{eq:admOrderDisjUnion}
		u_1,\dots,u_m,\ v_{1},\dots,v_{r},\ w_1,\dots,w_s
		\end{equation}
		is an admissible order of $\HS_k(I)$.
		
		Let $\ell\in\{1,\dots,m\}$. Then $(u_1,\dots,u_{\ell-1}):u_{\ell}$ is generated by variables.
		
		Let $\ell\in\{1,\dots,r\}$. We show that
		$$
		Q=(u_1,\dots,u_m,v_{1},\dots,v_{\ell-1}):v_{\ell}
		$$
		is generated by variables. Clearly $(v_1,\dots,v_{\ell-1}):v_{\ell}$ is generated by variables. Consider now $u_q:v_{\ell}$, $1\le q\le m$. Recall that $v_\ell$ is a monomial in $K[x_1,\dots,x_n]$. Thus $x_j$ divides $u_q:v_{\ell}$ for some $j\in\{n+1,\dots,n+m\}$. Consider $v_\ell/x_t$ for some $t$. Then $u=x_j(v_\ell/x_t)\in\HS_k(L)$ and $u:v_\ell=x_j\in Q$, as desired.
		
		Finally, let $\ell\in\{1,\dots,s\}$. We show that
		$$
		Q=(u_1,\dots,u_m,v_{1},\dots,v_{r},w_1,\dots,w_{\ell-1}):w_\ell
		$$
		is generated by variables. Since $w_1,\dots,w_{s}$ is an admissible order, $(w_1,\dots,w_{\ell-1}):w_{\ell}$ is generated by variables. Consider now a generator $z:w_\ell$ with $z=u_q$ or $z=v_q$, for some $q$. Since $w_\ell$ is a monomial in $K[x_{n+1},\dots,x_{n+m}]$, $z:w_\ell$ is divided by a variable $x_i$, where $i\in[n]$. Consider $w_\ell/x_t$ for some $t$. Then $u=x_i(w_\ell/x_t)\in\HS_k(L)$ and $u:w_\ell=x_i\in Q$, as desired.
		
		The three cases above show that (\ref{eq:admOrderDisjUnion}) is an admissible order, as desired. 
	\end{proof}
	
	\begin{proof}[Proof of Theorem \ref{Thm:HSForests}]
		Let $\Gamma=G^c$ be a forest and let $c$ be the number of connected components of $\Gamma$. If $c=1$, then $\Gamma$ is a tree, and by Corollary \ref{Cor:HSTree}, $I(G)$ has homological linear quotients. Suppose $c>1$ and write $\Gamma=\Gamma_1\cup\Gamma_2$, where $\Gamma_1$ and $\Gamma_2$ are disjoint forests. The the numbers of connected components of $\Gamma_1$ and $\Gamma_2$ are smaller than $c$. Thus, by induction $I(\Gamma_1^c)$ and $I(\Gamma_2^c)$ have homological linear quotients. Applying Proposition \ref{Prop:HSDisjoint}, it follows that $I(G)$ has homological linear quotients, too.
	\end{proof}

	Let $G$ be a complete multipartite graph, then $G^c$ is the disjoint union of some complete graphs. Repeated applications of Proposition \ref{Prop:HSDisjoint} yield
	\begin{Corollary}
		Let $G$ be a complete multipartite graph. Then $I(G)$ has homological linear quotients.
	\end{Corollary}
	
	\section{Polymatroidal ideals generated in degree two}
	
	A \textit{polymatroidal ideal} $I\subset S=K[x_1,\dots,x_n]$ is a monomial ideal $I$ generated in a single degree verifying the following \textit{exchange property}: for all $u,v\in G(I)$ with $u\ne v$ and all $i$ such that $\deg_{x_i}(u)>\deg_{x_i}(v)$, there exists $j$ such that $\deg_{x_j}(u)<\deg_{x_j}(v)$ and $x_j(u/x_i)\in G(I)$.
	
	The name polymatroidal ideal is justified by the fact that their minimal generating set corresponds to the set of bases of a \textit{discrete polymatroid}. A squarefree polymatroidal ideal is called \textit{matroidal}. Any polymatroidal ideal also satisfy a dual version of the exchange property.
	\begin{Lemma}\label{Lemma:SymExchProp}
		\textup{(\cite[Lemma 2.1]{HH2003}).} Let $I\subset S$ be a polymatroidal ideal. Then, for all $u,v\in G(I)$ and all $i$ such that $\deg_{x_i}(u)>\deg_{x_i}(v)$, there exists $j$ such that $\deg_{x_j}(u)<\deg_{x_j}(v)$ and $x_i(v/x_j)\in G(I)$.
	\end{Lemma}
	
	There are many useful characterization of polymatroidal ideals. The following one is due Bandari and Rahmati-Asghar.
	\begin{Theorem}\label{Thm:BanRam}
		\textup{(\cite[Theorem 2.4]{BanRam019}).} Let $I\subset S$ be a monomial ideal generated in a single degree. Then, $I$ is polymatroidal if and only if $I$ has linear quotients with respect to the lexicographic order induced by any ordering of the variables.
	\end{Theorem}

	It is expected by Bandari, Bayati and Herzog that the homological shift ideals $\HS_k(I)$ of a polymatroidal ideal $I$ are all polymatroidal, see \cite{Bay019,HMRZ021a}. In this section, we provide an affirmative answer to this conjecture for all polymatroidal ideals generated in degree two.
	
	Firstly, we deal with the squarefree case.
	\begin{Lemma}\label{Lemma:MatrOrders}
		Let $I\subset S$ be a matroidal ideal generated in degree two, and let $G$ be the simple graph on $[n]$ such that $I=I(G)$. Then, any ordering of the variables is a perfect elimination order of $G^c$.
	\end{Lemma}
	\begin{proof}
		Up to relabeling, we can consider the ordering $x_1>x_2>\dots>x_n$. Let $j,k\in N_{G^c}(i)$ with $j,k>i$. We must prove that $\{j,k\}\in E(G^c)$. By our assumption, $\{i,j\},\{i,k\}\notin E(G)$, that is $x_ix_j,x_ix_k\notin I(G)=I$. Suppose by contradiction that $\{j,k\}\notin E(G^c)$, then $\{j,k\}\in E(G)$, that is, $x_jx_k\in I(G)$. Pick any monomial $x_ix_s\in I(G)$. Then $\deg_{x_i}(x_ix_s)>\deg_{x_i}(x_jx_k)$. By Lemma \ref{Lemma:SymExchProp}, we can find $\ell$ with $\deg_{x_\ell}(x_ix_s)<\deg_{x_\ell}(x_jx_k)$ and $x_i(x_jx_k)/x_\ell\in I(G)$. Thus, either $x_ix_j\in I(G)$ or $x_ix_k\in I(G)$. This is a contradiction. Hence $\{j,k\}\in E(G^c)$, as desired.
	\end{proof}
	\begin{Corollary}\label{Cor:HSMatroidal}
		Let $I\subset S$ be a matroidal ideal generated in degree two. Then $\HS_{k}(I)$ is a matroidal ideal, for all $k\ge0$.
	\end{Corollary}
	\begin{proof}
		Let $G$ be the simple graph on $[n]$ such that $I=I(G)$. By Lemma \ref{Lemma:MatrOrders} and Theorem \ref{Thm:Froberg}, $G^c$ is a reversible chordal graph and any ordering of the variables is a reversible perfect elimination order of $G^c$. By Theorem \ref{Thm:HSI(G)LinQuot}, for all $k\ge0$, $\HS_{k}(I)$ has linear quotients with respect to the lexicographic order induced by any ordering of the variables. Thus, by Theorem \ref{Thm:BanRam}, $\HS_k(I)$ is matroidal, for all $k\ge0$.
	\end{proof}
	
	Now, we turn to the non-squarefree case.
	\begin{Theorem}\label{Thm:HSPolyDegree2}
		Let $I\subset S$ be a polymatroidal ideal generated in degree two. Then, $\HS_{k}(I)$ is a polymatroidal ideal, for all $k\ge0$.
	\end{Theorem}
	\begin{proof}
		If $I$ is squarefree, the thesis follows from Corollary \ref{Cor:HSMatroidal}. Suppose $I$ is not squarefree. Up to a suitable relabeling, we can write $I=(J,x_1^2,x_2^2,\dots,x_t^2)$, where $J$ is the squarefree part of $I$, \emph{i.e.}, $G(J)=\{u\in G(I):u\ \textup{is squarefree}\}$ and $1\le t\le n$. Then $J$ is a matroidal ideal. Let $G$ be the simple graph on $[n]$ with $J=I(G)$, then $G^c$ is cochordal. Let $u_1,\dots,u_m$ be an admissible order of $J$. We claim that
		$$
		u_1,\dots,u_m,x_1^2,x_2^2,\dots,x_t^2
		$$
		is an admissible order of $I$. We only need to prove that
		$$
		Q=(u_1,\dots,u_m,x_1^2,\dots,x_{\ell-1}^2):x_\ell^2=(J,x_1^2,\dots,x_{\ell-1}^2):x_{\ell}^2
		$$
		is generated by variables. Indeed, let $x_ix_j:x_\ell^2\in Q$ be a generator with $i\le j$. If $x_ix_j:x_\ell^2$ is a variable, there is nothing to prove. Otherwise $x_ix_j:x_\ell^2=x_ix_j$, and $\ell\ne i,j$. Since $\deg_{x_\ell}(x_\ell^2)>\deg_{x_\ell}(x_ix_j)$, by the exchange property, $w=x_k(x_\ell^2)/x_\ell=x_kx_\ell\in I$, with $k=i$ or $k=j$. Then $k\ne\ell$, $w=x_kx_\ell\in J$ and $w:x_\ell^2=x_k\in Q$ is a variable that divides $x_ix_j:x_\ell^2$, as desired.
		
		We claim that $\textup{set}(x_{\ell}^2)=[n]\setminus\{\ell\}$, for all $\ell=1,\dots,t$. Let $i\in[n]\setminus\{\ell\}$. Then $x_ix_j\in G(I)$ for some $j$. If $j=\ell$, then $x_ix_\ell\in I$. Suppose $j\ne\ell$, then $\deg_{x_j}(x_ix_j)>\deg_{x_j}(x_\ell^2)$. By the exchange property, $x_ix_\ell\in I$, as desired.
		
		By equation (\ref{eq:HSILinQuot}), for all $k>0$,
		\begin{align*}
		\HS_k(I)\ &=\ \HS_k(J)+\sum_{\ell=1}^tx_\ell^2\cdot\HS_{k-1}((x_i:i\in[n]\setminus\{\ell\})).
		\end{align*}
		
		We set $J_{\ell}=(x_i:i\in[n]\setminus\{\ell\})$, $\ell=1,\dots,t$. Since $J$ is matroidal, $\HS_{k}(J)$ is matroidal by Corollary \ref{Cor:HSMatroidal}. Moreover, each ideal $J_\ell$ is generated by variables, and so it is matroidal. Hence all ideals $x_\ell^2\cdot\HS_{k-1}(J_\ell)$ are polymatroidal.
		
		To verify that $\HS_{k}(I)$ is polymatroidal, we check the exchange property. Let $u,v\in G(\HS_{k}(I))$ and $i$ such that $\deg_{x_i}(u)>\deg_{x_i}(v)$.
		
		To achieve our goal, we note the following fact. Let $w\in S$ be any squarefree monomial of degree $k+1$ and let $\ell\in[t]$. Then $x_\ell w\in\HS_{k}(I)$. Indeed, if $x_\ell$ divides $w$, then $x_\ell w\in x_\ell^2\cdot\HS_{k-1}(J_\ell)\subset\HS_{k}(I)$. Suppose $x_\ell$ does not divide $w$. For all $i$ such that $x_i$ divides $w$, $x_ix_\ell\in J$ because $i\ne\ell$. Fix a lexicographic order $\succ$ such that $x_\ell>x_i$ for all $i\in[n]\setminus\ell$. Up to relabeling, we can assume $\ell=1$ and that $\succ$ is induced by $x_1>x_2>\dots>x_n$. Writing $x_\ell w=x_\ell x_{j_2}\cdots x_{j_{k+2}}$ with $\ell=1<j_2<\dots<j_{k+2}\le n$, then $x_\ell x_{j_{k+2}}\in J$, $x_\ell x_{j_i}\in J$ and $x_\ell x_{j_i}\succ x_\ell x_{j_{k+2}}$, for $i=2,\dots,k+1$. Hence
		$$
		\{j_2,\dots,j_{k+1}\}\subseteq\big\{j\ |\ x_j\in(u\in G(J):u\succ x_\ell x_{j_{k+2}}):x_{\ell}x_{j_{k+2}}\big\}.
		$$
		This shows that $x_\ell w\in\HS_k(J)\subset\HS_k(I)$, because by Theorem \ref{Thm:BanRam}, $J$ has linear quotients with respect to $\succ$.

		If $u,v\in\HS_{k}(J)$ or $u,v\in x_\ell^2\cdot\HS_{k-1}(J_\ell)$, we can find $j$ with $\deg_{x_j}(u)<\deg_{x_j}(v)$ such that $x_j(u/x_i)\in\HS_{k}(I)$, because both $\HS_{k}(J),x_\ell^2\cdot\HS_{k-1}(J_\ell)$ are polymatroidal.
		
		Suppose now $u\in\HS_{k}(J)$ and $v\in x_\ell^2\cdot\HS_{k-1}(J_\ell)$. Then $\deg_{x_\ell}(u)<\deg_{x_\ell}(v)$ and $x_\ell(u/x_i)\in\HS_{k}(I)$, because $u/x_i$ is a squarefree monomial of degree $k+1$.
		
		Suppose $u\in x_\ell^2\cdot\HS_{k-1}(J_\ell)$ and $v\in\HS_{k}(J)$. Let $j$ such that $\deg_{x_j}(u)<\deg_{x_j}(v)$. Then $\deg_{x_j}(u)=0$. If $i=\ell$, then $x_j(u/x_\ell)\in\HS_{k}(I)$ because it is the product of $x_\ell$ times a squarefree monomial of degree $k+1$. If $i\ne\ell$, then $x_j(u/x_i)$ can also be written as such a product. In any case $x_j(u/x_i)\in\HS_{k}(I)$.
		
		Finally, suppose $u\in x_\ell\cdot\HS_{k-1}(J_\ell)$ and $v\in x_h^2\cdot \HS_{k-1}(J_h)$ with $\ell\ne h$. Suppose $i=\ell$ and let $j$ such that $\deg_{x_j}(u)<\deg_{x_j}(v)$. Then $u'=x_j(u/x_i)$ is either $x_\ell$ times a squarefree monomial of degree $k+1$, or is equal to $x_h$ times a squarefree monomial of degree $k+1$. In both cases, $u'\in\HS_{k}(I)$. Suppose now $i\ne\ell$. If there exist more than one $j$ with $\deg_{x_j}(u)<\deg_{x_j}(v)$, we can choose $j\ne h$. Then $\deg_{x_j}(v)=1$ and so $x_j$ does not divide $u$. Consequently $x_j(u/x_i)$ is equal to $x_\ell$ times a squarefree monomial of degree $k+1$, and so $x_j(u/x_i)\in\HS_{k}(I)$. If there is only one $j$ such that $\deg_{x_j}(u)<\deg_{x_j}(v)$, then $j=h$. We claim that $x_h$ does not divide $u$, then $x_h(u/x_i)$ is equal to $x_\ell$ times a squarefree monomial of degree $k+1$, and so $x_h(u/x_i)\in\HS_{k}(I)$, as wanted. Writing $v=x_h^2x_{j_1}\cdots x_{j_k}$, with $j_p\in[n]\setminus\{h\}$, $p=1,\dots,k$, then $\deg_{x_{j_p}}(v)=1\le\deg_{x_{j_p}}(u)$, for all $p=1,\dots,k$. Then $x_{j_1}\cdots x_{j_k}$ divides $u/(x_ix_\ell)$ because $\deg_{x_\ell}(u)>1\ge\deg_{x_\ell}(v)$ and $\deg_{x_i}(u)=1>\deg_{x_i}(v)$. This implies that $u=x_ix_\ell\cdot x_{j_1}\cdots x_{j_k}$. From this presentation it follows that $x_h$ does non divide $u$, because $i,\ell\ne h$ and $j_p\ne h$ for $p=1,\dots,k$, as well. 
		
		The cases above show that the exchange property holds for all monomials of $G(\HS_{k}(I))$. Hence $\HS_{k}(I)$ is polymatroidal and the proof is complete.
	\end{proof}

\end{document}